\documentclass[10pt]{article} 

\usepackage[utf8]{inputenc} 

\usepackage{geometry} 
\usepackage{graphicx} 
\usepackage{amsmath} 
\usepackage{amsfonts} 
\usepackage{amsthm} 
\usepackage{amssymb} 
\usepackage{enumitem}

\usepackage{algorithm}
\usepackage[noend]{algpseudocode}
\usepackage{authblk}
\usepackage{eucal}
\usepackage{url}
\usepackage{csvsimple}
\usepackage{sectsty}
\usepackage{verbatim}
\usepackage{longtable}

\sectionfont{\fontsize{12}{15}\selectfont}

\newtheorem{thm}{Theorem}
\newtheorem{lem}[thm]{Lemma}

\title{Symmetry Implies Isomorphism for Certain Maximum Length Circuit Codes}
\author{Kevin M. Byrnes
\thanks{E-mail:\texttt{dr.kevin.byrnes@gmail.com}}
}

\begin{document}
\maketitle

\begin{abstract}
A classic result due to Douglas \cite{Douglas2} establishes that, for odd spread $k$ and dimension $d=\frac{1}{2}(3k+3)$, all maximum length $(d,k)$ circuit codes are isomorphic.  Using a recent result of Byrnes \cite{Byrnes2020} we extend Douglas's theorem to prove that, for $k$ even $\ge 4$ and $d=\frac{1}{2}(3k+4)$, all maximum length symmetric $(d,k)$ circuit codes are isomorphic.
\end{abstract}

Let $I(d)$ denote the graph of the $d$-dimensional hypercube.
A cycle $C$ of $I(d)$ is a \emph{$d$-dimensional circuit code of spread $k$}, also called a $(d,k)$ circuit code, if it satisfies the distance requirement:
\begin{equation}
\label{eq_1}
d_{I(d)}(x,y)\ge \min \{d_C(x,y),k\} \ \forall x,y\in C
\end{equation}
\noindent where $d_{I(d)}(x,y)$ and $d_C(x,y)$ denote the minimum path length between vertices $x$ and $y$ in $I(d)$ and $C$, respectively.
Computing the maximum length of a $(d,k)$ circuit code, $K(d,k)$, for a given dimension $d$ and spread $k$ is an extremely challenging computational problem, and significant analysis is required to make the problem tractable even for small values of $d$ and $k$ \cite{Kochut, Hood, Ostergard}.
Exact formulas for $K(d,k)$ (for particular infinite families of $(d,k)$ pairs) are exceedingly rare.
In a groundbreaking paper, Douglas \cite{Douglas2} (building upon the previous work of Singleton \cite{Singleton}) established the following formulas for $K(d,k)$.

\begin{thm}[\cite{Douglas2} Theorem 4]
\label{thm_1}
Let $k$ be odd and let $d=\frac{1}{2}(3k+3)$, then $K(d,k)=4k+4$.  Furthermore there is a unique, up to isomorphism of $I(d)$, $(d,k)$ circuit code of length $K(d,k)$.
\end{thm}

\begin{thm}[\cite{Douglas2} Theorem 3]
\label{thm_2}
Let $k$ be even and let $d=\frac{1}{2}(3k+4)$.  Then $K(d,k)=4k+6$.
\end{thm}

\begin{thm}[\cite{Douglas2} Theorem 5]
\label{thm_3}
Let $k$ be odd and $\ge 9$, and let $d=\frac{1}{2}(3k+5)$.  Then $K(d,k)=4k+8$.
\end{thm}
\noindent Note that isomorphism of all maximum length $(d,k)$ circuit codes is only established in the case where $k$ is odd and $d=\frac{1}{2}(3k+3)$, as in Theorem \ref{thm_1}.

While a circuit code, $C$, can be represented as a sequence of vertices $C=(x_1,x_2,\ldots,x_N)$ of $I(d)$, 
it is typically more convenient to express $C$ in another form.
Consider any bijection between the vertices of $I(d)$ and the set of binary vectors of length $d$ such that:
two vertices are adjacent in $I(d)$ if and only if their attendant vectors differ in a single position, and $x_1$ is mapped to $\vec{0}$.
Then we may equivalently describe a circuit code $C$ by its \emph{transition sequence} $T(C)=(\tau_1,\tau_2,\ldots,\tau_N)$
where $\tau_i$ denotes the single position in which the vectors attendant to $x_i$ and $x_{i+1}$ differ (with $x_{N+1}=x_1$ as $C$ is a cycle).
A \emph{transition} (or \emph{transition variable}) is a particular $\tau_i$, while the \emph{transition elements} are the unique values assumed by $\{\tau_1,\ldots,\tau_N\}$,
without loss of generality the set $[d]=\{1,2,\ldots,d\}$.
A transition sequence is \emph{symmetric} if $\tau_i = \tau_{N/2+i}$ for $i=1,\ldots,\frac{N}{2}$.  
We consider two $(d,k)$ circuit codes $C$ and $C'$ to be isomorphic if $T(C)$ can be transformed into $T(C')$ by a cyclic shift and a permutation of
$[d]$ (reflecting the fact that $C$ and $C'$ are identical up to the selection of the starting vertex and some symmetry of $I(d)$).
In this note we use the concept of a symmetric transition sequence to establish an isomorphism result analogous to Theorem \ref{thm_1} in the case where $k$ is even and $d=\frac{1}{2}(3k+4)$.

\begin{lem}
\label{lem_4}
Let $k$ be even and $\ge 4$ and let $d=\frac{1}{2}(3k+4)$.  Then the maximum length of a symmetric $(d,k)$ circuit code is $4k+6$.  Furthermore there is a unique, up to isomorphism of $I(d)$, symmetric $(d,k)$ circuit code of length $4k+6$.
\end{lem}

To establish Lemma \ref{lem_4} we require the concept of a bit run.  Recall (from \cite{Byrnes2020}) that a \emph{segment} of $T(C)=(\tau_1,\ldots,\tau_N)$ is a cyclically consecutive subsequence $\omega=(\tau_i,\tau_{i+1},\ldots,\tau_j)$ (with subscripts $>N$ reduced modulo $N$).
We say that a segment $\omega$ is a \emph{bit run} if all of the transitions in $\omega$ are distinct (i.e. assume distinct values).
Singleton \cite{Singleton} showed that if a $(d,k)$ circuit code is sufficiently long, then it contains a bit run of length $k+2$.

\begin{thm}[\cite{Singleton} Theorem 1]
\label{thm_5}
Let $C$ be a $(d,k)$ circuit code with transition sequence $T(C)$ and having length $|C|>2(k+1)$.  
Then for any segment $\omega$ of $T(C)$ with length $\ge k+3$, either the first or last $k+2$ transitions of $\omega$ are a bit run.
\end{thm}

For any segment $\omega$ of $T(C)$, define $\delta(\omega)$ as the number of transition elements appearing an odd number of times in $\omega$.  Suppose that $N=|C|>2k$ and $\omega=(\tau_i,\tau_{i+1},\ldots,\tau_{j-1})$, then we have the following inequalities (see \cite{Byrnes2020} equations (2)-(4)):

\begin{equation}
\label{eq_2}
\delta(\omega,\tau_j)=\delta(\omega)\pm 1,
\end{equation}

\begin{equation}
\label{eq_3}
\delta(\omega)=|\omega|, \text{ if } |\omega|\le k+1,
\end{equation}

\begin{equation}
\label{eq_4}
\delta(\omega)\ge k, \text{ if } k\le |\omega| \le N-k.
\end{equation}

\noindent Let $\mathcal{F}(d,k,k+l)$ denote the family of all $(d,k)$ circuit codes $C$ such that $T(C)$ contains a bit run of length $\ge k+l$,
and let $S(d,k,k+l)$ denote the maximum length of a symmetric circuit code in $\mathcal{F}(d,k,k+l)$.
We derive Lemma \ref{lem_4} as a corollary to a recent result of \cite{Byrnes2020}:

\begin{thm}
\label{thm_6}
Let $k$ and $l$ be integers $\ge 2$ of opposite parity with $k\ge 2l+1$ if $k$ is odd, and $k\ge 2l-2$ if $k$ is even, and let $d=\frac{1}{2}(3k+l+1)$.  Then:
(i) $S(d,k,k+l)=4k+2l$, and
(ii) for $l=2 \text{ or } 3$ there is a unique (up to isomorphism) symmetric circuit code in $\mathcal{F}(d,k,k+l)$ of length $S(d,k,k+l)$.
\end{thm}

\begin{proof}[\textbf{Proof of Lemma \ref{lem_4}}]
From Theorems \ref{thm_2} and \ref{thm_6} it immediately follows that the maximum length of a symmetric $(d,k)$ circuit code 
(for $k$ even $\ge 4$ and $d=\frac{1}{2}(3k+4)$) is $4k+6$.  
Furthermore, by Theorem \ref{thm_6} part (ii), to prove isomorphism of all maximum length symmetric $(d,k)$ circuit codes it suffices to show that for any such circuit code $C$, $T(C)$ contains a bit run of length $k+3$ (implying all such circuit codes are in $\mathcal{F}(d,k,k+3)$).
By Theorem \ref{thm_5}, $T(C)$ must contain a bit run of length $k+2$, so we may assume that $T(C)$ has the form:
\begin{equation}
\label{eq_5}
T(C)=(\underbrace{1,2,\ldots,k+2}_{\omega_1},x,\underbrace{\beta_1,\ldots,\beta_k}_{\omega_2},\omega_1,x,\omega_2).
\end{equation}

We begin with some preliminary observations.
By construction, $\omega_1$ is a bit run, and since $|(x,\omega_2)|=k+1$, the segment $(x,\omega_2)$ is also a bit run by (\ref{eq_3}).
Thus every transition element in $[d]$ appears at most twice in $(\omega_1,x,\omega_2)$ (at most once in each non-overlapping bit run), 
and since $|T(C)|=K(d,k)$ all $d$ transition elements appear at least once in $(\omega_1,x,\omega_2)$ (for if the transition element $t\in [d]$ were not present in $(\omega_1,x,\omega_2)$, then $T'=(\omega_1,x,\omega_2,t,\omega_1,x,\omega_2,t)$ would be a symmetric $(d,k)$ circuit code of length $4k+8$).
Define $\psi_i=(\beta_i,\ldots,\beta_k,1,\ldots,i)$ for $i=1,\ldots,k$, and define $\rho_j=(j+3,\ldots,k+2,x,\beta_1,\ldots,\beta_j)$ for $j=1,\ldots,k-1$.
Since $|\psi_i|=|\rho_j|=k+1 \ \forall i\in [k] \text{ and } \forall j \in [k-1]$, both $\psi_i$ and $\rho_j$ are bit runs, implying:
\begin{equation}
\label{eq_6}
\beta_i > i \text{ for } i=1,\ldots,k \text{ and } \beta_j \not\in \{j+3,\ldots,k+2\} \text{ for } j=1,\ldots,k-1.
\end{equation}

Now, $(3,\ldots,k+2,x)$ is a segment of length $k+1$ and so is a bit run.
If $x\not\in \{1,2\}$ then the segment $(\omega_1,x)$ of $T(C)$ is a bit run of length $k+3$, completing the proof.
Thus we will assume that $x\in \{1,2\}$.
First we establish that $(\omega_2,1,2)$ is a bit run.
Clearly $x\not\in \omega_2$ as $(x,\omega_2)$ is  bit run.
Consider the segment of $T(C)$, $\omega=(x,\omega_2,1,2)$, which has length $k+3$.
Since $|T(C)|=4k+6>2(k+1)$, by Theorem \ref{thm_5} either the first or last $k+2$ transitions of $\omega$ constitute a bit run.
If $x=1$ this implies that $(\omega_2,1,2)$ must be a bit run, 
while if $x=2$ then (since $(\omega_2,1)$ is a bit run by (\ref{eq_3}))
both $(x,\omega_2,1)$ and $(\omega_2,1,2)$ are bit runs.

If $3\not\in\omega_2$ then $(\omega_2,1,2,3)$ is a bit run of length $k+3$ in $T(C)$ and we are done, so we will suppose that $3\in \omega_2$.
From (\ref{eq_6}) it follows that $3\in \{\beta_1,\beta_2\}$.
Define $\omega$ as: $\omega=(x,\omega_2,1,2,3)$, then $|\omega|=k+4$.
Since $(\omega_2,1,2)$ is a bit run, both $x \text{ and } 3\in (\omega_2,1,2)$ by assumption, and $x\neq 3$, we have:
$\delta(\omega)=|\omega|-2\cdot|(\omega_2,1,2)\cap \{x,3\}|=k+4-2\cdot|\{x,3\}|=k$.
This implies $4\not\in \omega_2$, otherwise the segment $\omega'=(\omega,4)$ would have $\delta(\omega')=k-1$ (following from (\ref{eq_2})), which violates (\ref{eq_4}).
Now we split into cases depending upon whether $\beta_1=3$ or $\beta_2=3$ (note that since $\omega_2$ is a bit run exactly one of these alternatives holds).\\

\noindent\textbf{Case 1:} $\beta_1=3.$\\
Observe that $1,2,4\not\in \omega_2$ and $3\not\in \{\beta_2,\ldots,\beta_k\}$.  
Thus $(\beta_2,\ldots,\beta_k,1,2,3,4)$ is a bit run in $T(C)$ of length $k+3$.\\

\noindent \textbf{Case 2:} $\beta_2=3$.\\
Define $\omega=(4,\ldots,k+2,x,\beta_1,3,\beta_3)$, then $|\omega|=k+3$.
Since $x\in \{1,2\}$ and $x\not\in \omega_2$, $\omega$ fails to be a bit run only if $\beta_1$ or $\beta_3\in [4,k+2]=\{4,5,\ldots,k+2\}$.
By (\ref{eq_6}): $\beta_1\not\in [4,k+2]$, and if $\beta_3\in [4,k+2]$ then $\beta_3<6$.
As we have established that $1,2,4 \not\in \omega_2$ and $\beta_3\neq \beta_2=3$, the second implication can be sharpened to: $\beta_3=5$.
Therefore: $1,2,4 \not\in \omega_2$ and both $3$ and $5\not\in \{\beta_4,\ldots,\beta_k\}$.
In this case we claim that $\omega'=(\beta_4,\ldots,\beta_k,1,2,3,4,5,6)$ is a bit run of length $k+3$ in $T(C)$.
If not, then we must have $6\in \{\beta_4,\ldots,\beta_k\}$, but the segment $\omega''=(x,\omega_2,1,2,3,4,5,6)$ has length $k+7$ and the 
only repeated transition elements are: 1 xor 2, 3, 5, and 6 (each occurring twice in $\omega''$).
This means $\delta(\omega'')=|\omega''|-2\cdot|(x,\omega_2)\cap (1,\ldots,6)|=k+7-8=k-1$, violating (\ref{eq_4}).
Thus $6\not\in \{\beta_4,\ldots,\beta_k\}$ and so $\omega'$ is a bit run in $T(C)$ of length $k+3$.\\

In all cases, we have shown that $T(C)$ contains a bit run of length $k+3$, completing the proof.
\end{proof}

Before concluding, we note that the main technical result in the proof of Lemma \ref{lem_4} (the existence of a bit run of length $k+3$) also follows (after modification) from Case II of the proof of \cite{Douglas2} Theorem 3 (what we have labelled as Theorem \ref{thm_1}).
Specifically, there it is shown that for $k$ even and $d=\frac{1}{2}(3k+4)$ if a $(d,k)$ circuit code $C$ with $|C|\ge 4k+8$ exists, then $T(C)$ must contain a bit run of length $k+3$.
However, it appears possible to modify the proof to use the weaker condition $|C|\ge 4k+6$.
Since the main purpose of this note is to observe how all symmetric $(d,k)$ circuit codes (for $d$ and $k$ as in Lemma \ref{lem_4}) of length $4k+6$ are isomorphic as a consequence of Theorem \ref{thm_6}, rather than establishing the existence of a $k+3$ bit run in any (potentially asymmetric) such $(d,k)$ circuit code of length $\ge 4k+6$, we have presented an alternate, more accessible, proof.

\footnotesize
\bibliographystyle{plain}
\bibliography{SymmetryImpliesIsoForCC}

\begin{thebibliography}{1}

\bibitem{Byrnes2020}
K.~M. Byrnes.
\newblock The maximum length and isomorphism of circuit codes with long bit
  runs.
\newblock {\em https://arxiv.org/abs/2008.04839}, 2020.

\bibitem{Douglas2}
R.~J. Douglas.
\newblock Some results on the maximum length of circuits of spread k in the
  d-cube.
\newblock {\em Journal of Combinatorial Theory}, 6(4):323--339, 1969.

\bibitem{Hood}
S.~Hood, D.~Recoskie, J.~Sawada, and D.~Wong.
\newblock Snakes, coils, and single-track circuit codes with spread k.
\newblock {\em Journal of Combinatorial Optimization}, 30(1):42--62, 2013.

\bibitem{Kochut}
K.~J. Kochut.
\newblock Snake-in-the-box-code for dimension 7.
\newblock {\em Journal of Combinatorial Mathematics and Combinatorial
  Computing}, 20:175--185, 1996.

\bibitem{Ostergard}
P.R.J. \"{O}sterg\r{a}rd and V.H. Pettersson.
\newblock On the maximum length of coil-in-the-box codes in dimension 8.
\newblock {\em Discrete Applied Mathematics}, 179:193--200, 2014.

\bibitem{Singleton}
R.~C. Singleton.
\newblock Generalized snake-in-the-box codes.
\newblock {\em IEEE Trans. Electronic Computers}, 15:596--602, 1966.

\end{thebibliography}
\normalsize

\end{document}